\documentclass{llncs}

\usepackage{amsfonts}
\usepackage{amsmath}
\usepackage{amssymb}
\usepackage{amsthm}
\usepackage{xcolor}
\usepackage{hyperref}
\usepackage{multirow}
\usepackage{array}
\usepackage{mathtools}
\usepackage{bm}

\renewcommand{\vec}[1]{\mathbf{#1}}

\newcommand{\RR}{\mathbb{R}}

\newcommand{\ZZ}{\mathbb{Z}}
\newcommand{\NN}{\mathbb{N}}
\newcommand{\EE}{\mathbb{E}}

\newcommand{\dd}[1]{\mathop{d#1}}

\DeclareMathOperator{\vspan}{Span}
\DeclareMathOperator{\gind}{ind}
\DeclareMathOperator{\gdep}{dep}
\DeclareMathOperator{\loc}{loc}

\title{\textbf{Geometric Structure in Weighted Alpert Wavelets}}
\author{Fletcher Gates and Scott Rodney}
\institute{}
\date{March 26, 2025}

\begin{document}

\maketitle

\begin{abstract} 
In this paper we present a number of results concerning Alpert wavelet bases for $L^2(\mu)$, with $\mu$ a locally finite positive Borel measure on $\RR^n$. We show that the properties of such a basis depend on linear dependences in $L^2(\mu)$ among the functions from which the wavelets are constructed; this result completes an investigation begun by Rahm, Sawyer, and Wick \cite{rahm2019weighted}. We also show that a Gr\"{o}bner basis technique can be used to efficiently detect these dependences. Lastly we give a generalization of the Alpert basis construction, where the amount of orthogonality in the basis is allowed to vary over the dyadic grid.
\end{abstract}

\section{Introduction and Main Results}
\label{sec:introduction}

The classical Haar basis for $L^2(\RR)$ consists of wavelets supported on dyadic intervals which are piecewise constant and which are orthogonal to constants. In \cite{alpert1992}, Alpert constructs a similar basis where each wavelet is instead orthogonal to all polynomials with degree less than some fixed $k$. This is achieved by making two concessions: 
\begin{enumerate}
\item Where each Haar wavelet is piecewise constant, each Alpert wavelet is instead piecewise polynomial of degree less than $k$.
\item Where the Haar basis has one wavelet per dyadic interval, an Alpert basis has $k$ wavelets per dyadic interval.
\end{enumerate}
Moreover, Alpert showed that such a basis can be constructed so that some of the basis functions are orthogonal to higher-order polynomials. The algorithm to produce such a basis is given in \cite[Section 1.1]{alpert1992}.

Rahm, Sawyer, and Wick \cite{rahm2019weighted} showed that Alpert's construction also yields a basis for $L^2(\mu)$, where $\mu$ is any locally finite positive Borel measures on $\RR^n$. The resulting bases have been used to prove a number of results regarding Calder\'on-Zygmund operators---see for example \cite{alexis2023t1theorem} by Alexis, Sawyer, and Uriarte-Tuero, where Alpert bases are used to extend the David-Journ\'{e} $T1$ theorem \cite{david1984} from Lebesgue measure to pairs of doubling measures. Alpert wavelets were also recently used by Sawyer to prove a probabilistic analogue of the Fourier extension conjecture \cite{sawyer2024}.

Rahm, Sawyer, and Wick also observed that, in certain measures, the number of wavelets needed for each dyadic cube is sometimes less than would be needed in Lebesgue measure. We give a complete categorization of this behaviour in sections \ref{sec:dimension} and \ref{sec:polynomial-alpert-wavelets}. In section \ref{sec:variable-alpert-bases} we give a generalization of Alpert's construction, which allows for more control in the construction of the basis.

The remainder of this section will give ``high-level" but imprecise statements of our three main results. Section \ref{sec:preliminaries} defines the necessary concepts and notation to make each of these statements precise.

Let $\mu$ be a locally finite positive Borel measure on $\RR^n$, let $D$ be a dyadic grid on $\RR^n$, and let $U$ be a finite set of locally-integrable functions which contains the constant function $1$. The associated Alpert basis for $L^2(\mu)$ consists of functions which are each supported on some cube $Q \in D$, are each in $\vspan U$ when restricted to any child of $Q$, and are each orthogonal to all of $U$. In such a basis, the number of Alpert wavelets needed on a given cube $Q$ depends on both the set $U$ and the underlying measure $\mu$.

\newtheorem*{thm:dimension-inequality}{Theorem \ref{thm:dimension-inequality}}
\begin{thm:dimension-inequality}
When constructing an Alpert basis, the number of wavelets needed for a given dyadic cube $Q$ is the sum of the dimensions of $\vspan U$ when restricted to each child of $Q$, minus the dimension of $\vspan U$ when restricted to $Q$. Moreover, when applying extra orthogonality conditions to an Alpert basis, each additional condition is either achieved ``for free" without affecting the basis, or otherwise is achieved by reducing the number of wavelets satisfying that condition by 1.
\end{thm:dimension-inequality}

Our second result provides a further refinement of this result for Alpert bases constructed from polynomials of some bounded degree; this is the special case discussed by Rahm, Sawyer, and Wick in \cite{rahm2019weighted}. This result uses concepts from algebraic geometry with which the reader may be unfamilar; definitions are provided in section \ref{sec:preliminaries}.

\newtheorem*{thm:classification-theorem}{Theorem \ref{thm:classification-theorem}}
\begin{thm:classification-theorem}
Let $G$ be a Gr\"{o}bner basis for the ideal of all polynomials on $Q$ which vanish outside a set of $\mu$-measure zero. The set of all monomials on $Q$ which are not a multiple of any leading term in $G$ is a maximal linearly independent set.
\end{thm:classification-theorem}

By Theorem \ref{thm:dimension-inequality}, knowing the dimensions of these polynomial spaces is sufficient to determine the number of Alpert wavelets needed for a given cube. This result has additional utility in that a single Gr\"{o}bner basis calculation simultaneously provides the dimensions of these polynomial spaces for every choice of degree.

Our third result presents a generalization of the standard Alpert basis construction. For a dyadic cube $Q$, let $P(Q)$ denote the parent of $Q$.

\newtheorem*{thm:variable-alpert-bases}{Theorem \ref{thm:variable-alpert-bases}}
\begin{thm:variable-alpert-bases}
For every dyadic cube $Q$ let $U_Q$ be a finite set of locally-integrable functions, and suppose that these sets obey the nesting condition $U_{P(Q)} \subseteq U_Q$. Then an orthonormal basis for $L^2(\mu)$ can be constructed from the following three components:
\begin{enumerate}
\item For every $Q \in D$, a set of Alpert wavelets on $Q$ constructed from $U_Q$.
\item For every $Q \in D$, an orthonormal basis for the orthogonal complement of $\vspan{U_{P(Q)}}$ inside $\vspan{U_Q}$.  
\item For every dyadic top $T \in \tau(D)$, an orthonormal basis for $\vspan U_T$, where $U_T = \cap_{Q \subset T} U_Q$.
\end{enumerate}
This basis retains the same orthogonality properties as a standard Alpert basis.
\end{thm:variable-alpert-bases}

A basis of this type allows for interpolation between the amount of orthogonality achieved and the number of wavelets needed for each cube.

\section{Preliminaries}
\label{sec:preliminaries}

We define a \textit{cube} $Q \in \RR^n$ to be specifically oriented with sides parallel to the coordinate hyperplanes, i.e. a set of the form 
\[Q = \{x \in \RR^n : a_i \leq x_i < b_i, i= 1, \dots, n\} = \prod_{i=1}^n [a_i, b_i)\]
where each $a_i, b_i \in \RR$ is constant and where there is a constant $l(Q)>0$ such that
\[b_1 - a_1 = b_2 - a_2 = \dots = b_n - a_n = l(Q)\]
We call $l(Q)$ the side length of $Q$.

A \textit{dyadic grid} $D$ on $\RR^n$ is a set of cubes with the following properties:
\begin{enumerate}
\item Every cube $Q \in D$ has $l(Q) = 2^m$ for some $m \in \ZZ$.
\item Every cube $Q \in D$ with $l(Q) = 2^m$ is contained in some $R \in D$ with $l(R) = 2^{m+1}$.
\item For every $m \in \ZZ$, the subset of all cubes in $D$ with length $2^m$ partitions $\RR^n$. 
\end{enumerate}
Given a dyadic grid $D$ and $Q \in D$ with side length $2^m$, the set of children $C(Q)$ is the set of $2^n$ cubes in $D$ contained within $Q$ having side length $2^{m-1}$. Similarly the parent $P(Q)$ is the unique cube in $D$ which contains $Q$ and has side length $2^{m+1}$

Let $\mu$ be a locally finite positive Borel measure on $\RR^n$, $D$ be a dyadic grid on $\RR^n$, and $Q \in D$. Denote by $L_Q^2(\mu)$ the space of functions in $L^2(\mu)$ restricted to $Q$. Also let $L^2_{\loc}(\mu)$ be the space of locally square-integrable functions on $\RR^n$, so we have $L^2_Q(\mu) \subset L^2_{\loc}(\mu)$ for every $Q \in D$. We will be investigating the following two spaces:

\begin{definition}[Component Space]
Let $\mu$ be a locally finite positive Borel measure on $\RR^n$, $Q \in D$, and $U \subset L^2_{\loc}(\mu)$. Define the component space $P_{Q,U}(\mu) = \vspan\{\vec 1_Q \cdot p\}_{p \in U}$, the subspace of $L_Q^2(\mu)$ generated by the restrictions of $U$ to $Q$.
\end{definition}

\begin{definition}[Alpert Space]
Let $\mu$ be a locally finite positive Borel measure on $\RR^n$, $Q \in D$, and $U, V \subset L^2_{\loc}(\mu)$. Define the Alpert space $L^2_{Q, U, V}(\mu)$ to be the subspace of functions in $L_Q^2(\mu)$ whose restrictions to each child $Q' \in C(Q)$ are in $\vec 1_{Q'} U$ and which are orthogonal to each function in $V$. Namely:
\[L^2_{Q, U, V}(\mu) = \left\{f = \sum_{Q' \in C(Q)} \vec 1_{Q'} p: \int_Q f(x) \cdot q(x) \dd{\mu (x)} = 0 \text{ for all } q(x) \in V \right\}\]
where each function $p$ is in $U$.
\end{definition}

Alpert spaces can be thought of as the spaces from which individual Alpert wavelets are selected to form a basis. The following identity is helpful to visualize the structure of an Alpert space in terms of its component spaces:
\[L^2_{Q, U, \varnothing}(\mu) = \bigoplus_{Q' \in C(Q)} P_{Q',U}(\mu) = L^2_{Q, U, V}(\mu) \oplus P_{Q,U}(\mu). \]

Denote by $\RR[\vec{x}]$ the ring of real polynomials in $n$ variables, with $n$ being taken from context. Given $S \in \RR[\vec{x}]$, the \textit{zero locus} $Z(S) \subseteq \RR^n$ is the set of common roots to every $p \in S$. An \textit{algebraic set} $A$ in $\RR^n$ is any subset of $\RR^n$ which is the zero locus of some $S$, i.e. $A = Z(S)$ for some $S \subset \RR[\vec x]$.

We will use the notation $x^\alpha$ to denote an arbitrary monic monomial in $n$ variables, so $\alpha$ is a multi-index of length $n$. A \textit{monomial order} $M$ on $\RR[\vec{x}]$ is a well-ordering on the set of all monic monomials which respects multiplication, and $M$ is \textit{graded} if the ordering also respects total degree. For a polynomial $p \in \RR[\vec{x}]$, the leading term $LT(p)$ is the greatest monomial in $p$ with respect to $M$. Similarly for any set $S$ of polynomials, $LT(S)$ is the set of leading terms $\{LT(p)\}_{p \in S}$.

\begin{definition}
Given $k,n \in \NN$, define $F^n_k$ to be the set of all monomials in $n$ variables with degree less than $k$.
\end{definition}

By a standard combinatorial argument, we have $\#F^n_k = {n + k - 1 \choose n}$. This is the number of ways to partition $k-1$ into $n+1$ groups, and by excluding the final group we see that this is equal to the number of ways to partition $k-1$ into at most $n$ groups. 

\begin{definition}[Gr\"{o}bner Basis]
Let $M$ be a monomial order on $\RR^n$ and let $I$ be an ideal in $\RR[\vec{x}]$. A Gr\"{o}bner basis $G$ for $I$ is a generating set for $I$ such that for any polynomial $p \in I$, $LT(p)$ is divisible by $LT(q)$ for some $q \in G$. A Gr\"{o}bner basis is called reduced if no monomial in any $p \in G$  is in $LT(G \setminus \{p\})$ and every $p \in G$ is monic. 
\end{definition}

An equivalent definition is: a Gr\"{o}bner basis $G$ is a generating set for $I$ such that $\langle LT(I) \rangle = \langle LT(G) \rangle$. While Gr\"{o}bner bases depend on the choice of monomial order and in general are not unique, they are guaranteed to exist for polynomial rings over a field in finitely many variables. Moreover, for a given monomial order there is a unique reduced Gr\"{o}bner basis. Given an ideal $I \in \RR[\vec x]$ there exist efficient algorithms to compute a Gr\"{o}bner basis for $I$; see for example \cite[pp.~324-326]{dummitfoote}.

\begin{definition}[Hilbert Dimension]
Let $I$ be an ideal in $\RR[\vec x]$. The Hilbert dimension of $I$ is the maximal size of a subset $S$ of variables in $\vec x$ such that no leading monomial in $I$ can be expressed entirely using variables in $S$. 
\end{definition}

Observe that Hilbert dimension can be easily computed using a Gr\"{o}bner basis; given an ideal $I$ with Gr\"{o}bner basis $G$, the set of leading terms $LT(I)$ is equal to $LT(G)$. So the size of a maximal set of variables which produces none of the elements in $LT(G)$ gives the Hilbert dimension of $I$.

Lastly, we define the following two sets:

\begin{definition}
Let $\mu$ be a locally finite positive Borel measure on $\RR^n$, $D$ be a dyadic grid, and $Q \in D$. Define $A_Q$ to be the intersection of all algebraic sets $A$ such that $\mu(Q \setminus A) = 0$.
\end{definition}

\begin{definition}
Let $\mu$ be a locally finite positive Borel measure on $\RR^n$, $D$ be a dyadic grid, and $Q \in D$. Define $I_Q$ to be the ideal of all polynomials in $\RR[\vec x]$ which vanish on $A_Q$. 
\end{definition}

Informally, $A_Q$ and $I_Q$ give a canonical representation of all the linear dependences among the monomials inside $L^2_Q(\mu)$. Computing a generating set for $I_Q$ given $A_Q$ is a difficult problem, in the sense that there is no general algorithm for producing such sets and individual cases must be solved heuristically. However, provided that we can find a generating set for $I_Q$, we can then leverage a Gr\"{o}bner basis to compute a basis for $P_{Q,F^n_k}(\mu)$.

\section{Dimensions of Alpert Spaces}
\label{sec:dimension}

Let $\mu$ be a locally finite positive Borel measure on $\RR^n$, $D$ be a dyadic grid, and $U \subseteq L^2_{\loc}(\mu)$. The standard construction of an Alpert basis selects, for every $Q \in D$, an orthonormal basis for $L^2_{Q, U, U}(\mu)$ and the union of each of these bases gives the Alpert basis for $L^2(\mu)$. \footnote{There is one additional wrinkle concerning dyadic tops which we have omitted here for clarity of exposition. The complete explanation is given in section \ref{sec:variable-alpert-bases}.}

It is, however, useful to uncouple the set of component functions from the set of orthogonality conditions; this is the motivation behind our definition of $L^2_{Q, U, V}(\mu)$, where $U$ and $V$ are both taken to be finite sets in $L^2_{\loc}(\mu)$. This allows us to construct an Alpert space by starting with $L^2_{Q, U, \varnothing}(\mu)$ and then imposing orthogonality to the elements of $V$ one function at a time.

\begin{lemma}
\label{thm:dimension-thm}
Let $\mu$ be a locally finite positive Borel measure on $\RR^n$, $D$ be a dyadic grid, $Q \in D$, $U, V \subseteq L^2_{\loc}(\mu)$ be finite sets, and $f \in L^2_{\loc}(\mu)$ be a function such that $f \notin V$.  Then either $\dim{L^2_{Q, U, V \cup \{f\}}(\mu)} = \dim{L^2_{Q, U, V}(\mu)}$, if $f$ is orthogonal to all of $L^2_{Q, U, V}(\mu)$, or $\dim{L^2_{Q, U, V \cup \{f\}}(\mu)} = \dim{L^2_{Q, U, V}(\mu)} - 1$ otherwise.
\end{lemma}

\begin{proof}
If $f$ is orthogonal to all of $L^2_{Q, U, V}(\mu)$ then $L^2_{Q, U, V \cup \{f\}}(\mu)$ and $L^2_{Q, U, V}(\mu)$ are the same space and trivially have the same dimension. Suppose instead that $f$ is not orthogonal to all of $L^2_{Q, U, V}(\mu)$. Then let $\{b_1,...,b_k\}$ be a basis for $L^2_{Q, U, V}(\mu)$ and suppose without loss of generality that $f$ is not orthogonal to $b_1$.

Now for $i=2,...,k$, define the constant $c_i \in \RR$ as
\[c_i = -\frac{\int_Q b_i(x)f(x)\dd{\mu(x)}}{\int_Q b_1(x)f(x)\dd{\mu(x)}}.\]
This gives
\[\left\langle b_i + c_ib_1, f \right\rangle = \int_Q b_i(x)f(x) \dd{\mu(x)}  - \frac{\int_Q b_i(x)f(x)\dd{\mu(x)}}{\int_Q b_1(x)f(x)\dd{\mu(x)}} \int_Q b_1(x)f(x)\dd{\mu(x)} = 0\]
so $b_i + c_ib_1$ is orthogonal to $f$. Now suppose that $\{b_i+c_ib_1\}_{i=2,...,k}$ contains a linear dependence. Then we would have
\[\sum_{i+2}^k a_i (b_i+c_ib_1) = 0\]
for some constants $a_i \in \RR$, $i=2,...k$ not all zero. Rearranging gives
\[ \left(\sum_{i+2}^k a_i c_i \right)b_1+ \sum_{i+2}^k a_i b_i = 0\]
which is impossible as $\{b_1,...,b_k\}$ is a basis for $L^2_{Q, U, V}(\mu)$. Consequently $\{b_i+c_ib_1\}_{i=2,...,k}$ is linearly independent and forms a basis for $L^2_{Q, U, V \cup \{f\}}(\mu)$, and we conclude that $\dim{L^2_{Q, U, V \cup \{f\}}(\mu)} = \dim{L^2_{Q, U, V}(\mu)} - 1$.
\end{proof}

The key insight here is that $L^2_{Q, U, \varnothing}(\mu)$ is a finite-dimensional vector space, and that any $L^2_{Q,U,V}(\mu)$ is a subspace of $L^2_{Q, U, \varnothing}(\mu)$. In particular, suppose that $V = \{v_1,...v_k\}$ is a finite set of functions none of which are orthogonal to $L^2_{Q, U, \varnothing}(\mu)$. Then by Lemma \ref{thm:dimension-thm}, each $L^2_{Q,U, \{v_i\}}(\mu)$ is a hyperplane inside $L^2_{Q, U, \varnothing}(\mu)$, and
\[L^2_{Q,U,V}(\mu) = \bigcap_{i=1,...,k} L^2_{Q,U, \{v_i\}}(\mu).\]

Since $V$ is finite by definition, applying Lemma \ref{thm:dimension-thm} iteratively to every function in $V$ yields upper and lower bounds for the dimension of $L^2_{Q, U, V}(\mu)$.

\begin{theorem}
\label{thm:dimension-inequality}
Let $\mu$ be a locally finite positive Borel measure on $\RR^n$, $D$ be a dyadic grid, $Q \in D$, and $U, V \subseteq L^2_{\loc}(\mu)$ be finite sets. Then
\[ \sum_{Q' \in C(Q)} \dim P_{Q',U}(\mu) - \dim P_{Q,V}(\mu) \leq \dim{L^2_{Q, U, V}(\mu)} \leq \sum_{Q' \in C(Q)} \dim P_{Q',U}(\mu).\]
\end{theorem}

\begin{proof}
Let $V_0 \subseteq V$ be a maximal subset such that $\vec 1_Q V_0$ is linearly independent. Suppose that $f \in L^2_{\loc}(\mu)$ is a function which can be expressed as a linear combination of the elements in $V_0$. Since $L^2_{Q, U, V_0}(\mu)$ is orthogonal to all of $V_0$, it is also orthogonal to $f$. Then by applying the first conclusion of Lemma \ref{thm:dimension-thm} to each $f \in V \setminus V_0$, we conclude that $L^2_{Q, U, V}(\mu)$ and $L^2_{Q, U, V_0}(\mu)$ are the same space.

Now since $\dim P_{Q,V}(\mu) = \dim P_{Q,V_0}(\mu) = \#V_0$, applying Lemma \ref{thm:dimension-thm} to each element in $V_0$ gives 
\[ \dim{L^2_{Q, U, \varnothing}(\mu)} - \#V_0 \leq \dim{L^2_{Q, U, V_0}(\mu)} \leq \dim{L^2_{Q, U, \varnothing}(\mu)}\]
and consequently
\[ \sum_{Q' \in C(Q)} \dim P_{Q',U}(\mu) - \dim P_{Q,V}(\mu) \leq \dim{L^2_{Q, U, V}(\mu)} \leq \sum_{Q' \in C(Q)} \dim P_{Q',U}(\mu)\]
as desired.
\end{proof}

The following example shows that both sides of the inequality in Theorem \ref{thm:dimension-inequality} are achievable, and therefore that finding the dimension of an Alpert space requires investigation of the underlying geometry of $\mu$ as it relates to $U$ and $V$.

\begin{example}
\label{ex:extra-orthogonality}
Take $\mu$ to be Lebesgue measure on $\RR$, and $D$ to be a dyadic grid containing the interval $I = [-1,1)$. Let $I_l$ and $I_r$ denote the left and right subintervals of $I$. If we take $U = V_0 = \{1\}$ then $L^2_{I,U,V_0}(\mu)$ is the usual one-dimensional Haar space spanned by $\vec 1_{I_l} - \vec 1_{I_r}$. If we now expand the set of orthogonality conditions to $V_1 = \{1,x\}$, we see that  $L^2_{I,U,V_0}(\mu)$ contains no non-trivial functions orthogonal to $x$ and so $L^2_{I,U,V_1}(\mu)$ has dimension zero.

Instead, if we take $V_2 = \{1,x^2\}$ we see that the Haar functions in $L^2_{I,U,V_0}(\mu)$ are already orthogonal to $x^2$---indeed they are orthogonal to every even function on $\RR$. Therefore $L^2_{I,U,V_0}(\mu)$ and $L^2_{I,U,V_2}(\mu)$ are the same space and have the same dimension, despite the additional orthogonality condition.
\end{example}

When it comes to constructing a particular Alpert basis, this is not a serious problem. The algorithm for constructing Alpert bases---see \cite{alpert1992} for its details---already requires finding a non-orthogonal basis function for each orthogonality condition to be introduced. Any ``freebies" among the set of orthogonality conditions will be discovered in the course of performing the algorithm, and Theorem \ref{thm:dimension-thm} guarantees that there are no other kinds of bad behaviour to be concerned about.

Lastly, we emphasize that these considerations are only a concern when the additional orthogonality properties allowed by Alpert bases are of interest; this is equivalent to asking that a basis for $L^2_{Q, U, U}(\mu)$ additionally have some elements belonging to $L^2_{Q, U, V}(\mu)$ for some $V \supseteq U$. If any basis for $L^2_{Q, U, U}(\mu)$ will suffice, the situation is much simpler.

\begin{corollary}
\label{thm:simple-dimension-equality}
Let $\mu$ be a locally finite positive Borel measure on $\RR^n$, $D$ be a dyadic grid, $Q \in D$, and $U, V \subseteq L^2_{\loc}(\mu)$ be finite sets such that $V \subseteq U$. Then
\[\dim{L^2_{Q, U, V}(\mu)} = \sum_{Q' \in C(Q)} \dim P_{Q',U}(\mu) - \dim P_{Q,V}(\mu).\]
\end{corollary}

\begin{proof}
Since $V \subseteq U$, we have $P_{Q,V} \subseteq L^2_{Q, U, \varnothing}(\mu)$. $P_{Q,V}$ is therefore the orthogonal complement of $L^2_{Q, U, V}(\mu)$ inside $L^2_{Q, U, \varnothing}(\mu)$ as finite-dimensional vector spaces, and the result follows immediately.
\end{proof}

So the problem of finding dimensions and bases for Alpert spaces $L^2_{Q, U, U}(\mu)$ reduces to the equivalent problem for component spaces $P_{Q,U}(\mu)$. A maximal set of linearly independent functions in $U$ gives a basis for $P_{Q,U}(\mu)$, so given $Q$, $U$, and $\mu$ it suffices to find such a set. A brute force algorithm could proceed as follows: begin with the entire set $U$ and iteratively apply the Gram Schmidt process to identify and eliminate dependent functions until a basis is found. This is rather inefficient, so in section \ref{sec:polynomial-alpert-wavelets} we will give a more sophisticated algorithm when working with polynomial Alpert bases.

\section{Polynomial Alpert Bases}
\label{sec:polynomial-alpert-wavelets}

In section \ref{sec:dimension}, we saw that the dimension of an Alpert space $L^2_{Q, U, V}(\mu)$ depended non-trivially on the underlying geometry of both the measure $\mu$ and the set of functions $U, V$ in question. We now turn our attention to the specific case of polynomial Alpert bases; the additional structure this affords will allow us to draw more precise conclusions. We consider spaces $L^2_{Q,U,U}(\mu)$ where $U$ is taken to be some $F^n_k$, the set of all monomials in $n$ variables up to some fixed degree. 

In \cite{rahm2019weighted}, Rahm, Sawyer, and Wick observed that these polynomial Alpert spaces over certain measures had lower dimension than the corresponding spaces would have over Lebesgue measure. By Lemma \ref{thm:simple-dimension-equality} we have
\[\dim{L^2_{Q, F^n_k, F^n_k}(\mu)} = \sum_{Q' \in C(Q)} \dim P_{Q',F^n_k}(\mu) - \dim P_{Q,F^n_k}(\mu).\]
This behaviour is therefore explained by observing that, unlike in Lebesgue measure, the monomials in $F^n_k$ are not guaranteed to be linearly independent and consequently that $\dim P_{Q,F^n_k}(\mu)$ might be less than $\#F^n_k$.

This reduces the question of finding $\dim{L^2_{Q, F^n_k, F^n_k}(\mu)}$ to the question of finding $\dim P_{Q,F^n_k}(\mu)$ for arbitrary $Q$. It is not immediately obvious that this can be done easily; even if some subset of $F^n_k$ is linearly independent in $L^2(\mu)$, it might be linearly dependent in $L^2_Q(\mu)$ for some $Q$.

Recall that for a given $Q \in D$ we have defined $A_Q$ to be the intersection of all algebraic sets $A$ such that $\mu(Q \setminus A) = 0$, and $I_Q$ to be ideal of all polynomials in $\RR[\vec x]$ which vanish on $A_Q$. Provided that we can find a generating set for $I_Q$, we can use a Gr\"{o}bner basis to find a basis for $P_{Q,F^n_k}(\mu)$.

\begin{theorem}
\label{thm:classification-theorem}
Let $\mu$ be a locally finite positive Borel measure on $\RR^n$, $D$ be a dyadic grid, $Q \in D$, and $k \in \NN$. Also let $M$ be a graded monomial order and $G$ be the reduced Gr\"{o}bner basis for $I_Q$.  Define $\gdep_k(G)$ to be the set of all monomials $u \in F^n_k$ which are divisible by some monomial $v \in LT(G)$, and define $\gind_k(G)$ to be the complement of $\gdep_k(G)$ in $F^n_k$. Then $\gind_k(G)$ is a basis for $P_{Q,F^n_k}(\mu)$, and consquently $\dim{P_{Q,F^n_k}(\mu)} = \#F^n_k - \#\gdep_k(G)$.
\end{theorem}

\begin{proof}
Suppose $\gind_k(G)$ contains a linear dependence in $L_Q^2(\mu)$, given by some polynomial $p(x)=0$. Then $p \in I_Q$, so $LT(p) \in LT(G)$ and $LT(p) \notin \gind_k(G)$. This is a contradiction, so $\gind_k(G)$ is linearly independent in $L_Q^2(\mu)$ and it remains to show that $\gind_k(G)$ is maximal.

Let $u_0 \in \gdep_k(G)$ and consider the set $T \coloneqq \{u_0\} \cup \gind_k(G)$. Since $u_0 \in LT(G)$, there is some monomial $v_0 \in F^n_k$ and some polynomial $p_0 \in G$ such that $u_0 = LT(v_0 p_0)$. Then since the monomial order $M$ is graded and $LT(v_0 p_0)$ has degree less than $k$, all monomials in $v_0 p_0$ must have degree less than $k$ and so $v_0 p_0$ is a linear dependence in $F^n_k$. It now suffices to show that $v_0 p_0$ can be expressed using only monomials in $T$.

Let $u_1 \notin T$ be a non-leading monomial occurring in $v_0 p_0$. Since $u_1 \in \gdep_k(G)$ there is some monomial $v_1 \in F^n_k$ and some polynomial $p_1 \in G$ such that $u_1 = LT(v_1 p_1)$. The non-leading monomials in $v_1 p_1$ all have order less than $u_1$, so we can replace $u_1$ in $v_0 p_0$ with lower order monomials. Since $\gdep_k(G)$ is finite, iterating this process a sufficient number of times yields a representation for $v_0 p_0$ using only monomials in $T$. Consequently $T$ is linearly dependent in $L_Q^2(\mu)$ for any choice of $u_0 \in \gdep_k(G)$, so we conclude that $\gind_k(G)$ must be a basis for $P_{Q,F^n_k}(\mu)$.
\end{proof}

Our statement of this theorem is framed in terms of a particular component space $P_{Q,F^n_k}(\mu)$, but we highlight that the single Gr\"{o}bner basis $G$ encodes the necessary information regarding this space for all values $n \in \NN$. This is the primary contribution of Theorem \ref{thm:classification-theorem}, since finding the dimension of any single $P_{Q,F^n_k}(\mu)$ could otherwise be done with a brute force Gram-Schmidt algorithm. Building on this idea, we can also use $G$ to determine how the dimension of $P_{Q,F^n_k}(\mu)$ grows---and therefore also how $\dim{L^2_{Q, F^n_k, F^n_k}(\mu)}$ grows---as $k$ increases.

\begin{lemma}
\label{thm:unbounded-growth-lemma}
Let $I$ be an ideal in $\RR[\vec x]$ with Hilbert dimension $d$. Then there exists $k \in \NN$ such that $x^\alpha \in LT(I)$ for all $\alpha \in \NN^n$ with more than $d$ entries greater than $k$.
\end{lemma}

\begin{proof}
Suppose toward a contradiction that there is no such $k$. Then for every $k \in \NN$ we can associate an $x^{\alpha_k} \notin LT(I)$ where at least $d+1$ entries in $\alpha_k$ are greater than $k$. This yields the sequence of monomials $\{x^{\alpha_k}\}_{k \in \NN}$. Since there are only finitely many variables to choose from, there must be a set $S = \{x_{i_1}, x_{i_2}, \dots, x_{i_{d+1}}\}$ of $d+1$ distinct variables such that for any $l \in \NN$ the product $x_{i_1}^l \cdot x_{i_2}^l \cdots x_{i_{d+1}}^l$ divides some monomial in $\{x^{\alpha_k}\}_{k \in \NN}$. Since no $x^{\alpha_k}$ is in $LT(I)$, we conclude that $x_{i_1}^l \cdot x_{i_2}^l \cdots x_{i_{d+1}}^l \notin LT(I)$ for all $l \in \NN$.

Next we see that any product of the variables in $S$ divides $x_{i_1}^l \cdot x_{i_2}^l \cdots x_{i_{d+1}}^l$ for some sufficiently large choice of $l$, and consequently any product of these variables produces a monomial not in $LT(I)$. Therefore $S$ shows that $I$ must have Hilbert dimension at least $d+1$. This completes the contradiction, so such a $k$ must exist.
\end{proof}

We now have the following intuition: by Theorem \ref{thm:classification-theorem}, the leading terms of a Gr\"{o}bner basis for $I_Q$ determine the linearly dependent monomials that need to be excluded to form a basis for $P_{Q,F^n_k}(\mu)$. After making all of these exclusions, what remains should grow like a Lebesgue polynomial function space. While this does not allow us to compute $\dim P_{Q,F^n_k}(\mu)$ directly, it does characterize the growth rate of this dimension as $k$ increases. Here we use big $O$ notation in its usual meaning: $f(x) = O(g(x))$ if for some constant $C>0$ and some $x_0$ we have $|f(x)| > C \cdot g(x)$ for all $x > x_0$.

\begin{proposition}
\label{thm:hilbert-dimension-theorem}
Let $\mu$ be a locally finite positive Borel measure on $\RR^n$, $D$ be a dyadic grid, $Q \in D$, and $k \in \NN$. Suppose that $I_Q$ has Hilbert dimension $d$. Then $\dim P_{Q,F^n_k}(\mu) = O(k^d)$.
\end{proposition}

\begin{proof}
If $d=n$ we are immediately done, so suppose that $d<n$. For any set $S'$ of $d+1$ variables, consider all monomials that can be expressed using only variables in $S'$. At least one such monomial must be a leading monomial within $I_Q$, otherwise this set would give $I_Q$ a Hilbert dimension of at least $d+1$. Thus we cannot have an independent monomial containing arbitrarily high powers of all $d+1$ variables in $S'$. For any choice of variable to have a bounded power, there are $O(k^d)$ such monomials. As there are only finitely many variables to choose, the total number of independent monomials using only variables in $S'$ is at most $O(k^d)$.

Now we generalize this argument to any set of more than $d$ variables. By lemma \ref{thm:unbounded-growth-lemma}, the largest number of variables which can all have arbitrarily large powers and still multiply to an independent monomial is $d$. For any choice of $d$ variables the remaining $n-d$ variables must all have bounded powers, and the total number of such monomials is $O(k^d)$. There are finitely many ways to choose $d$ variables, so even if every such monomial were independent and each choice of $d$ variables yielded a disjoint set of independent monomials, we would have at most $O(k^d)$.
\end{proof}

While this theorem tells us that the growth rate of $\dim P_{Q,F^n_k}(\mu)$ (as a function of $k$) is polynomial of degree $d$, it is not possible to give a general upper bound on the leading coefficient of that polynomial. To see this, take the ideal generated by $x_{n-d}^t \cdot x_{n-d+1}^t \cdots x_{n}^t$ for some $t \in \NN$. This ideal has Hilbert dimension $d$ for any choice of $t$, but the number of monomials outside this ideal grows with $t$.

\section{Variable Alpert Bases}
\label{sec:variable-alpert-bases}

The standard construction of an Alpert basis chooses some set of orthogonality conditions to be applied to the basis functions on each dyadic cube. In this section we show that this can be relaxed: each cube $Q$ can be given a differing set of orthogonality conditions, and with a small modification the resulting functions still form an orthonormal basis for $L^2(\mu)$.

Let $D$ be a dyadic grid on $\RR^n$ and $Q \in D$. We define the \textit{tower} $\Gamma(Q)$ as the set of all cubes in $D$ which contain $Q$. The \textit{top} $T$ of a tower $\Gamma(Q)$ is defined as the countable union of all cubes in the tower. As an immediate consequence of the nesting property of towers, any two towers will either have the same top or two disjoint tops. This gives an equivalence relation on the towers in $D$, where $\Gamma_1$ and $\Gamma_2$ are equivalent if their tops coincide. Let $\tau(D)$ denote the set of unique tops arising from any choice of representatives from each equivalence class. 

Alexis, Sawyer, and Uriarte-Tuero in \cite{alexis2022tops} observed that, besides bases for each $L^2_{Q,U,U}$, an Alpert basis may also require the restrictions to some dyadic tops of functions in $U$. Specifically, for every $f \in U$ and $T \in \tau(D)$, if $\vec 1_T f$ has finite $L^2(\mu)$-norm then it must be included in the Alpert basis. This remains true for our generalization.

Let $\EE^\mu_{Q,U}$ denote orthogonal projection onto $P_{Q,U}(\mu)$. With a slight abuse of notation we will allow a dyadic top $T \in \tau(D)$ in place of the dyadic cube $Q$, so $P_{T,U}(\mu)$ is the space of restrictions of $U$ to $T$ and $\EE^\mu_{T,U}$ is the associated projection. Let $\triangle^\mu_{Q, U, V}$ similarly denote orthogonal projection onto $L^2_{Q,U,V}(\mu)$.

For sets of functions $U \subseteq V$ we will also need to consider the orthogonal complement of $P_{Q,U}(\mu)$ inside $P_{Q,V}(\mu)$. As our notation is already somewhat cumbersome, we will write $P_{Q,V}(\mu) \ominus P_{Q,U}(\mu)$ to denote such a subspace and $\EE^\mu_{Q,V} - \EE^\mu_{Q,U}$ to denote the corresponding projection. Lastly, recall that $P(Q)$ denotes the parent of $Q$.

\begin{theorem}
\label{thm:variable-alpert-bases}
Let $\mu$ be a locally finite positive Borel measure on $\RR^n$, $D$ be a dyadic grid, and $\{U_Q\}_{Q \in D}$ be a collection of finite sets in $L^2_{\loc}(\mu)$ and having the following properties:
\begin{enumerate}
\item $1 \in U_Q$ for every $Q \in D$.
\item $U_Q \subseteq U_{Q'}$ for every $Q \in D$ and every $Q' \in C(Q)$.
\end{enumerate}
For each dyadic top $T \in \tau(D)$, let $U_T = \cap_{Q \subset T} U_Q$. Then
\[ \left\{ \triangle^\mu_{Q, U_Q, U_Q}\right\}_{Q \in D} \cup \left\{ \EE^\mu_{Q, U_Q} - \EE^\mu_{Q, U_{P(Q)}} \right\}_{Q \in D} \cup \left\{\EE^\mu_{T, U_T}\right\}_{T \in \tau(D)} \]
is a complete set of orthogonal projections in $L^2(\mu)$ and
\[f=\sum_{Q \in D} \triangle^\mu_{Q, U_Q, U_Q} f + \sum_{Q \in D} \left(\EE^\mu_{Q, U_Q} - \EE^\mu_{Q, U_{P(Q)}}\right)f + \sum_{T \in \tau(D)} \EE^\mu_{T, U_T} f , \qquad f \in L^2(\mu)\]
where convergence holds both in $L^2(\mu)$ and pointwise $\mu$-almost everywhere. Moreover we have telescoping identities,
\[ \vec 1_Q \sum_{P: Q \subsetneqq P \subseteq R} \triangle^\mu_{P, U_P, U_P} = \EE^\mu_{Q, U_Q} - \vec 1_Q \EE^\mu_{R, U_R} \text{  for all } Q \subsetneqq R \in D \text{ with } U_Q = U_R,\] 
and orthogonality conditions
\[\int_{\RR^n} \triangle^\mu_{Q, U_Q, U_Q} f(x) p(x) \dd{\mu(x)}=0, \quad \text{for all } Q \in D, p \in U_Q.\]
\end{theorem}

Here the $\triangle^\mu_{Q, U_Q, U_Q}$ are the usual Alpert projections, and the $\EE^\mu_{T, U_T}$ are the aforementioned projections on the tops of $D$. The $\EE^\mu_{Q, U_Q} - \EE^\mu_{Q, U_{P(Q)}}$ are the necessary addition which allows the basis to retain completeness while varying $U_Q$, which we prove now. Besides this consideration, we otherwise follow the strategy used by Rahm, Sawyer, and Wick in \cite[Theorem 1]{rahm2019weighted} to show the equavialent result for polynomial Alpert bases.

\begin{proof}
Fix $Q \in D$. By construction we have
\begin{equation}
L^2_{Q, U_Q, U_Q}(\mu) \oplus P_{Q,U_Q}(\mu) = L^2_{Q, U_Q, \varnothing}(\mu).
\end{equation}
Applying this to each child $Q' \in C(Q)$ we have
\[\bigoplus_{Q' \in C(Q)} P_{Q', U_{Q'}}(\mu) \oplus \bigoplus_{Q' \in C(Q)} L^2_{Q', U_{Q'}, U_{Q'}}(\mu) = \bigoplus_{Q' \in C(Q)}L^2_{Q', U_{Q'}, \varnothing}(\mu).\]
The leftmost sum in this expression can be rewritten as
\[\bigoplus_{Q' \in C(Q)} P_{Q', U_{Q'}}(\mu) = L^2_{Q, U_Q, U_Q}(\mu) \oplus P_{Q,U_Q}(\mu) \oplus \bigoplus_{Q' \in C(Q)} \left( P_{Q',U_{Q'}} \ominus P_{Q',U_Q}(\mu) \right).\]
The nesting property $U_Q \subseteq U_{Q'}$ guarantees that this construction is valid.

Next consider the tower $\Gamma(Q)$ and choose some $S \in \Gamma(Q)$. Let $X$ be the set of all cubes $R$ contained in $S$ and with side lengths $l(Q) < l(R) \leq l(S)$. Applying the above argument iteratively to every $P \in X$ yields
\[\bigoplus_{\substack{Q' \subset S \\ l(Q') = l(Q)}} L^2_{Q', U_{Q'}, \varnothing}(\mu) = \bigoplus_{R \in X} L^2_{R, U_R, U_R}(\mu) \oplus P_{S,U_S}(\mu) \oplus \bigoplus_{R \in X} \left( P_{R,U_R} \ominus P_{R,U_P(R)}(\mu) \right).\] 
In particular, this sum contains the Haar space $L^2_{Q, \{1\}, \{1\}}(\mu)$ since $1 \in U_Q$. 

Now take the limit of this construction as $S$ tends to infinity. Let $T$ be the top of $\Gamma(Q)$; in the limit, $P_{S,U_S}(\mu)$ becomes $P_{T,U_T}(\mu)$. Thus for any $Q \in D$ we have
\[L^2_{Q, \{1\}, \{1\}}(\mu) \subseteq \bigoplus_{Q \in D} L^2_{Q, U_Q, U_Q}(\mu) \oplus \bigoplus_{T \in \tau(D)} P_{T,U_T}(\mu) \oplus \bigoplus_{Q \in D} \left( P_{Q,U_Q} \ominus P_{Q,U_P(Q)}(\mu) \right).\] 
We know that the Haar spaces $L^2_{Q, \{1\}, \{1\}}(\mu)$, $Q \in D$, together with projections on the tops $P_{T, \{1\}}(\mu)$, $T \in \tau(D)$, form a direct sum decomposition of $L^2(\mu)$. We also have $P_{T, \{1\}}(\mu) \subseteq P_{T, U_T}(\mu)$ since $1 \in U_Q$ for every $Q \in D$, so we conclude
\[L^2(\mu) = \bigoplus_{Q \in D} L^2_{Q, U_Q, U_Q}(\mu) \oplus \bigoplus_{T \in \tau(D)} P_{T,U_T}(\mu) \oplus \bigoplus_{Q \in D} \left( P_{Q,U_Q} \ominus P_{Q,U_P(Q)}(\mu) \right).\]

The orthogonality conditions are satisfied by construction: the nesting property $U_Q \subseteq U_{Q'}$ ensures that any $\triangle^\mu_{Q, U_Q, U_Q} f$ is orthogonal to every $p \in U_R$ for any $R \in D$ containing $Q$. Lastly the telescoping property follows by chaining together instances of (1), exactly as in the standard Alpert construction.
\end{proof}

This construction partially alleviates one of the drawbacks of Alpert bases, namely that extra orthogonality is achieved only by making the basis much larger than its Haar counterpart. This might be of interest in applications where the additional orthogonality is only needed locally, or only beyond a particular level of resolution. We end with a pair of remarks concerning the structure of the resulting basis. 

\begin{remark}
The telescoping identity in Theorem \ref{thm:variable-alpert-bases} is weaker than in a standard Alpert basis where all $U_Q$ are equal. However since we still have $U_R \subseteq U_Q$ for $Q \subseteq R$, multiple instances of this weaker identity can be chained together to produce an expression that is still ``good". At each level in the chain one subtracts projections only onto the functions which are in $U_{P'}$ but not in $U_P$, for $P' \in C(P)$.
\end{remark}

\begin{remark}
The projections at the tops can be interpreted as a special case of the projections $\EE^\mu_{Q, U_Q} - \EE^\mu_{Q, U_{P(Q)}}$. If we think of a top $T$ as being a kind of dyadic cube and define $U_{P(T)} = \varnothing$, as tops do not have parents, then $\EE^\mu_{Q, U_{P(T)}}$ is trivial and we recover the usual projection on $T$. In this sense our generalization only extends an existing complexity in Alpert bases, rather than introducing a new one.
\end{remark}

\end{document}